\documentclass[12pt]{amsart}

\usepackage{fullpage,amsmath}
\usepackage{amsthm}
\usepackage{amssymb}
\usepackage{ mathrsfs }
\usepackage{bbm}
\usepackage{bm}
\usepackage{ stmaryrd }
\usepackage{tikz-cd}
\usepackage{amsmath,amssymb}

\definecolor{hot}{RGB}{65,105,225}
\usepackage[pagebackref=true,colorlinks=true, linkcolor=hot ,  citecolor=hot, urlcolor=hot]{hyperref}

\usepackage{float}
\usepackage{graphicx}
\usepackage{wrapfig}
\usepackage{caption}
\usepackage{url}
\usepackage{enumitem}
\usepackage{lipsum}
\usepackage{cleveref}
\crefformat{equation}{(#2#1#3)}

\usepackage{scalerel}

\newcommand{\ExtMinSpace}{\mkern -1mu}
\newcommand{\RelMinspace}{\mkern-5mu}

\renewcommand{\sc}[1]{\scaleto{#1}{8pt}}

\newcommand{\C}{\ensuremath{\mathbb{C}}}

\newcommand{\D}{\ensuremath{\mathscr{D}}}

\newcommand{\U}{\ensuremath{\mathscr{U}}}
\renewcommand{\O}{\ensuremath{\mathcal{O}}}
\newcommand{\M}{\ensuremath{\mathscr{M}}}
\newcommand{\N}{\ensuremath{\mathscr{N}}}
\newcommand{\blank}{{-}}

\newcommand{\gr}{\ensuremath{\operatorname{gr}}}
\newcommand{\grrel}{\ensuremath{\gr^{rel}\RelMinspace}}
\newcommand{\Ch}{\ensuremath{\operatorname{Ch}}}
\newcommand{\Chrel}{\Ch^{rel}\RelMinspace}
\newcommand{\abs}[1]{\ensuremath{\vert #1 \vert}}

\DeclareMathOperator{\Spec}{Spec}
\DeclareMathOperator{\Ann}{Ann}
\DeclareMathOperator{\LCT}{LCT}
\DeclareMathOperator{\RLCT}{RLCT}
\DeclareMathOperator{\lct}{lct}
\DeclareMathOperator{\ord}{ord}
\DeclareMathOperator{\KLT}{LCT}
\DeclareMathOperator{\RKLT}{RLCT}

\newcommand{\Ext}{\ensuremath{\mathcal{E}\ExtMinSpace xt}}

\newcommand{\Hom}{\ensuremath{\mathcal{H}om}}
\newcommand{\Ex}[2]{\Ext_{\sc{#1}}^{#2}}

\def\bR{\mathbb{R}}
\def\bZ{\mathbb{Z}}
\def\bQ{\mathbb{Q}}
\def\bN{\mathbb{N}}
\def\bC{\mathbb{C}}
\def\pa{\partial}
\def\sD{\mathscr{D}}
\def\ra{\rightarrow}
\def\lam{\lambda}
\def\al{\alpha}
\def\be{\begin{equation}}
\def\ee{\end{equation}}
\def\bk{\mathbf k}
\def\cO{\mathcal O}
\def\L{\mathscr{L}}
\def\cJ{\mathcal J}

\newtheorem{theorem}{Theorem}[section]
\newtheorem{corollary}[theorem]{Corollary}
\newtheorem{lemma}[theorem]{Lemma}
\newtheorem{proposition}[theorem]{Proposition}

\theoremstyle{definition}
\newtheorem{definition}[theorem]{Definition}
\newtheorem{remark}[theorem]{Remark}

\title[Estimates for zero loci of Bernstein-Sato ideals]{Estimates for zero loci of Bernstein-Sato ideals}

\author{Nero Budur}\address{KU Leuven, Department of Mathematics, Celestijnenlaan 200B, 3001 Leuven, Belgium, and BCAM, Mazarredo 14, 48009 Bilbao, Spain}
\email{nero.budur@kuleuven.be}

\author{Robin van der Veer}\address{KU Leuven, Department of Mathematics,  Celestijnenlaan 200B, 3001 Leuven, Belgium}\email{robinvdveer@gmail.com}

\author{Alexander Van Werde}\address{KU Leuven, Department of Mathematics,  Celestijnenlaan 200B, 3001 Leuven, Belgium}
\curraddr{Eindhoven University of Technology, Department of Mathematics and Computer Science, PO Box 513, 5600 MB Eindhoven, Netherlands}
\email{a.van.werde@tue.nl}

\keywords{Bernstein-Sato ideal, $\sD$-module, log resolution.}
\subjclass{14F10, 32C38, 32S45.}

\begin{document}

\begin{abstract}
We give estimates for the zero loci of Bernstein-Sato ideals.
An upper bound is proved as a multivariate generalisation of the upper bound by Lichtin for the roots of Bernstein-Sato polynomials.
The lower bounds generalise the fact that log-canonical thresholds, small jumping numbers of multiplier ideals, and their real versions provide roots of Bernstein-Sato polynomials.
\end{abstract}

\maketitle

\section{Introduction}
Let $F = (f_1,\ldots,f_r)$ with $f_j\in \bC[x_1,\ldots,x_n]$ be a tuple of polynomials, and $r>0$. Introduce new variables $s = (s_1,\ldots, s_r)$ and fix a tuple of natural numbers $a = (a_1,\ldots, a_r) \in \bN^r$ such that the product $f_1^{a_1}\ldots f_r^{a_r}$ admits zeros on $X=\bC^n$.
By definition, the {\it Bernstein-Sato ideal} $B_F^{a}$  consists of all polynomials $b(s)\in \C[s]$ such that
$$b(s) F^s \in \D_X[s]F^{s+a}$$
where $F^s = f_1^{s_1}\cdots f_r^{s_r}$,  $\D_X=\bC[x]\langle\pa\rangle$ is the ring of algebraic differential operators on $X$, with $x=x_1,\ldots,x_n$, $\pa=\pa_1,\ldots,\pa_n$, and $\pa_i=\pa/\pa x_i$ for $i=1,\ldots,n$. Here $\D_X[s]F^{s+a}$ is the $\sD_X[s]$-submodule of the free $\bC[x,f^{-1},s]$-module $\bC[x,f^{-1},s]F^{s+a}$ obtained by applying formally the operators in $\sD_X[s]$ to the symbol $F^{s+a}$ by using the usual derivation rules, where $f=f_1\ldots f_r$. The zero locus of the ideal $B_F^a$ is denoted $$Z(B_F^a)\subseteq \C^r.$$
This construction extends easily to the case when $F:X\to \bC^r$ is a morphism from a smooth affine complex algebraic variety, and also, by using analytic differential operators, to the case when $F:(X,x)\to (\bC^r,0)$ is the germ of a holomorphic map of complex manifolds. The latter are the so-called {\it local Bernstein-Sato ideals} $B_{F,x}^a$, and the former, $B_F^a$, equal the intersection of all local $B_{F,x}^a$ for $x$ in the zero locus of $f$. In the classical case $r=1=a$, the ring $\C[s]$ is a principal ideal domain and the unique monic generator $b_f(s)$ of $B_F^{1}$ is called the {\it Bernstein-Sato polynomial} of $F=f$. The Bernstein-Sato ideals measures in some sense the singularities of the mapping $F$, since, for example, $B_{F,x}^{\,\mathbf{1}}=\langle (s_1+1)\ldots(s_r+1)\rangle$ if and only if $F:(X,x)\to (\bC^r,0)$ is smooth by \cite[Proposition 1.2]{BM}, where $\mathbf{1}=(1,\ldots,1)$. One has:

\begin{theorem} \label{thrmMoreA}  (\cite[Theorem 1.1.1]{budur2020zeroII}) Let $F=(f_1,\ldots,f_r):X\ra\bC^r$ be a morphism of smooth complex affine irreducible algebraic varieties, or the germ at $x\in X$ of a holomorphic map on a complex manifold.  Let $a\in\bN^r$ such that $\prod_{j=1}^rf_j^{a_j}$ is not invertible as a holomorphic function on $X$. Then:
\begin{enumerate}
\item Every irreducible component of $Z(B_F^{a})$ of codimension 1 is a  hyperplane of type $l_1s_1+\ldots+l_rs_r+b=0$ with $l_j\in\bQ_{\ge 0}$, $b\in\bQ_{>0}$, and for each such hyperplane there exists $j$ with $a_j\ne 0$ such that $l_j>0$.
\item Every irreducible component of  $Z(B_F^{a})$ of codimension $>1$  can be translated by an element of $\bZ^r$ inside a component of codimension 1.
\end{enumerate}
\end{theorem}

For $r=1$ statement (2) is vacuous and (1) is equivalent to the classical result that the roots of the Bernstein-Sato polynomial $b_f$ are negative rational numbers, due to Kashiwara \cite{kashiwara1976b}. The first part without the strict positivity of $l_j$ is due to Sabbah \cite{sabbah1987proximite} and Gyoja \cite{gyoja1993bernstein}. The second part for the case $a=\mathbf{1}$ is due to Maisonobe \cite{maisonobe2016filtration}, a completely different proof of which was given recently by van der Veer \cite{robin}.

The first purpose of this paper is to further refine part (1) of the above theorem in terms of numerical data from log resolutions.
Let $\mu:Y\to X$ be a strong log resolution of $f$.
This means that $\mu$ is a projective morphism that is an isomorphism over the complement of $D$, the divisor defined by $f$, such that $Y$ is smooth and $\mu^*D$ is a simple normal crossings divisor.
The numerical data we refer to are the orders of vanishing $\ord_{E}(f_j)\in\bN$ of $f_j$ along  irreducible components $E$ of $\mu^*D$, and the orders of vanishing $k_E=\ord_{E}(\det \operatorname{Jac}(\mu))\in\bN$ of the determinant of the Jacobian of $\mu$, also equal to the coefficients of the relative canonical divisor $K_\mu$ of $\mu$. We show:

\begin{theorem}\label{thm: MainTheorem}
  Every irreducible component of $Z(B_F^a)$ of codimension $1$ is a hyperplane of the form
  $$\ord_{E}(f_1)s_1 + \cdots + \ord_{E}(f_r)s_r + k_E + c=0 $$
with $c\in\bZ_{>0}$.
\end{theorem}

Without the term $k_E$,  the statement was proven for $r=1$ by Kashiwara \cite{kashiwara1976b} and for $r\ge 1$ by \cite[Lemma 4.4.6]{budur2020zeroII}. The case $r=1$ of Theorem \ref{thm: MainTheorem} is due to Lichtin \cite{lichtin1989poles}, a new proof of which was given by Dirks-Musta\c{t}\u{a} \cite{DM}. 

If $r=1$, the upper bound $c< (n+a-1)N_E-k_E$ for $c$ as in Theorem \ref{thm: MainTheorem} can be deduced from  \cite[Theorem 0.4]{S-m}. For $r>1$ the problem of finding an upper bound for $c$ is open. In some cases this is known, e.g. \cite[Theorem 1]{Ma-free}, \cite[Theorem 1.9]{Ba}.

The second part of this paper contains a number of lower bounds for the Bernstein-Sato zero locus.
Firstly, one has an easy multivariate generalisation for the fact that the Bernstein-Sato polynomial $b_f(s)$, which corresponds to the case $r=1=a$, always has $-1$ as a root.

\begin{proposition}\label{prop1.3}
  Let $C$ be an irreducible component of $D$ such that $m:=\sum_{j=1}^{r}\ord_{C}(f_j)a_j\neq 0$.
  Then $\left(\sum_{j=1}^r \ord_{C}(f_j)s_j \right)+ c = 0$ determines an irreducible component of $Z(B_F^a)$ for $c = 1,\ldots, m$.
  \end{proposition}
Further, we generalise the fact that the jumping numbers of $f$ in $(0,\lct(f)+1)$ are roots of $b_{f}(s)$, \cite{ClassicalJump}, \cite[Theorem 2]{BSArbitraryVariety}. Recall that the log-canonical threshold $\lct(f)$ is the smallest jumping number of $f$.

For any $\lambda\in \mathbb{R}_{\geq 0}^r$ the {\it mixed multiplier ideal sheaf} of $F^\lambda$ is given by
$$\mathcal{J}(F^\lambda) = \mu_*\O_Y(K_\mu - \lfloor \sum_{j=1}^r \lambda_j \mu^* D_j \rfloor)$$
where $D_i$ denotes the divisor determined by $f_i$ and $\lfloor \blank \rfloor$ is the round-down of an $\mathbb{R}$-divisor.
Associated to $\lambda$ is the subset, which together with the induced Euclidean topology we call {\it region}, $$\mathcal{R}_F(\lambda):= \{\lambda'\in \mathbb{R}_{\geq 0}^r:\mathcal{J}(F^{\lambda}) \subseteq \mathcal{J}(F^{\lambda'})\}.$$
We note that $\mathcal{R}_F(\lambda)\subseteq \mathcal{R}_F(\lambda')$ if $\lambda_i\le \lambda'_i$ for all $i$.
The {\it jumping walls of $F$} are given by the intersection of $\mathbb{R}^r_{>0}$ with the boundary of $\mathcal{R}_F(\lambda)$ for some $\lambda$.
In the case $r=1$, these are the jumping numbers of $f$.

By the definition of mixed multiplier ideals, each facet of a jumping wall, that is, a codimension-one face, is cut out by a hyperplane of the form
$\sum_{j=1}^r\ord_E(f_j)s_j = k_E + c$
with $c\in \mathbb{Z}_{>0}$ and $E$ an irreducible component of $\mu^*D$. Thus facets of jumping walls can potentially determine  irreducible components of $Z(B_F^a)$ by replacing $s_j$ with $-s_j$, by Theorem \ref{thm: MainTheorem}.

The log-canonical threshold, or rather the interval $[0,\lct (f)]$, is generalised by the {\it $\LCT$-polytope}
$$ \LCT(F) :=  \bigcap_{E}\{\lambda\in \mathbb{R}_{\geq 0}^r : \sum_{j=1}^r \ord_E(f_j)\lambda_j \leq k_E + 1\}.$$
The facets of $\LCT(F)$ intersecting $\bR_{>0}^r$ non-trivially are always jumping walls of $F$.

Define the {\it $\KLT_{a}$-region}
$$\KLT_{a}(F):= \bigcap_{E}\{\lam\in \mathbb{R}_{\geq 0}^r: \sum_{j=1}^r \ord_E(f_j)(\lambda_j - a_j) < k_E + 1\}.$$
If $r=1$ then $\LCT_a(f)=[0,\lct(f)+a)$.

We  rephrase $\LCT(F)$ and $\LCT_a(F)$ in terms of log-canonical and Kawamata log-terminal singularities in \ref{sec: JumpingWall}. This shows that  $\LCT(F)$ and $\LCT_a(F)$ are independent of the chosen resolution.

\begin{theorem}\label{thm: JumpingWall} If a facet of a jumping wall of $F$ intersects $\KLT_{a}(F)$, then the facet determines an irreducible component of $Z(B_F^a)$.
\end{theorem}
This theorem was shown by \cite{cassou2011multivariable} for $Z(B_F^{\mathbf{1}})$ when $f_1,\ldots, f_r$ are germs of plane curves.
We employ the same method, which is essentially the one used in \cite{kollar1997singularities}, \cite[Theorem B]{ClassicalJump}.

From \Cref{thm: JumpingWall} we deduce a generalisation for the fact that the largest root of the Bernstein-Sato polynomial $b_f(s)$ is equal to $-\lct(f)$ when $r=1$.

\begin{corollary}\label{thm: LCT}
  Let $\sum_{j=1}^r \ord_E(f_j)s_j = k_E + 1$ define the affine span of a facet of $\LCT(F)$. Then $\sum_{j=1}^r \ord_{E}(f_j)s_j +k_E + 1=0$ defines an irreducible component of $Z(B_F^a)$ if there exists at least one $j$ with $a_j\neq 0$ and $\ord_{E}(f_j)\neq 0$.
\end{corollary}
This together with \Cref{thm: MainTheorem} implies the analogue of the maximality statement from the case $r=1=a$: the irreducible components of codimension one of $Z(B_F^a)$ originating from the $\LCT$-polytope are the closest to the origin with that slope. 

In the hyperplane arrangement case, by using the canonical log resolution for hyperplane arrangements, and picking the component E of the exceptional divisor corresponding to the origin so that $k_E + 1$ is the rank of the arrangement, Corollary \ref{thm: LCT}  recovers \cite[Theorem 5.6]{Lei}: 
$\{s_1 +\cdots+s_r +n=0\}\subseteq Z(B_F^{\mathbf{1}})$.


Saito \cite{RealLogCan} also introduced a version of log-canonical thresholds and jumping numbers for real algebraic functions, called {\it real log-canonical threshold} and {\it real jumping numbers}. ``Real" here refers to working over $\bR$. Real jumping numbers, like the usual jumping numbers defined when the base field is $\bC$, are positive rational numbers. It is shown in \cite{RealLogCan} that the negatives of small real jumping numbers are roots of  Bernstein-Sato polynomials.
Interesting about these real jumping numbers is that they do not have to agree with the usual jumping numbers.
These results are of further interest due to applications to statistics \cite{Wa}.

Mixed multiplier ideals and their jumping walls will be defined on real algebraic manifolds in \ref{sec: RealJump}.
There are also the associated notions of a {\it $\RKLT_a$-region}, {\it $\RLCT$-polytope} and real Bernstein-Sato ideal. In Theorem \ref{thmRs} and Corollary \ref{corRs} we give the real analogs of Theorem \ref{thm: JumpingWall} and Corollary \ref{thm: LCT}, generalising Saito's results.


For the proof of Theorem \ref{thm: MainTheorem} we follow the strategy of  Kashiwara \cite{kashiwara1976b} and Lichtin \cite{lichtin1989poles}.  The main problem for the case $r>1$ is that the $\sD_X$-modules computing the Bernstein-Sato ideals are not holonomic anymore, and thus a new idea is needed. This is essentially the problem which has been surmounted using relative holonomic $\sD$-modules first in \cite{maisonobe2016filtration}, and then in \cite{budur2020zeroI}, \cite{budur2020zeroII},  \cite{robin} in order to provide a topological interpretation of $Z(B_F^a)$. The results of these papers are thus crucial for us.  Relative holonomic $\sD$-modules appeared as early as \cite{sabbah1987proximiteII} and are also recently studied in  \cite{AnalyticDirectIm, FS}.  Our main technical result is Lemma \ref{gradeBound}. The proofs of the other results mentioned above are straight-forward and need no essential new ideas.

In Section \ref{sec: DXS} we gather the main results on relative $\sD$-modules we need. In Section \ref{secUB} we use these results to prove Theorem \ref{thm: MainTheorem}. In Section \ref{secLB} we prove the other results. In Section \ref{secEx} we give an example.

\medskip
\noindent
{\bf Acknowledgement.} We were informed that L. Wu  has also proven all the results in this paper. We thank D. Bath, L. Wu, and P. Zhao for useful discussions, and the referee for many good comments. N. Budur was supported by the grants  FWO G097819N, FWO G0B3123N,  Methusalem METH/15/026. R. van der Veer was supported by an FWO PhD fellowship.

\section{$\D_X[s]$-modules}\label{sec: DXS}

This section provides preliminaries on the theory of $\D_X[s]$-modules such as direct images and homological properties, where $s=(s_1,\ldots,s_r)$.

\subsection{Relative holonomic $\sD$-modules}\label{sec: RelHol}
Let $X$ be a smooth complex variety and let $R$ be a regular commutative finitely generated $\C$-algebra integral domain. The {\it sheaf of relative differential operators} on $X$ is defined by
$$\D_X^R := \D_X \otimes_\C R.$$
The order filtration $F_j \D_X$ on $\D_X$ extends to a filtration $F_j \D_X^R := F_j\D_X\otimes_\C R$ on $\D_X^R$.
The graded objects for this filtration are denoted by $\grrel$.
Denote $\pi_{T^*X}:T^*X \to X$ and $\pi_{\Spec R}: \Spec R \to \{\operatorname{pt}\}$ the projection maps onto $X$ and a point, respectively.
Since $\gr\D_X \cong (\pi_{T^*X})_* \O_{T^*X}$ \cite[Chapter 2]{hotta2007d} it holds that $\grrel\D_X^R \cong (\pi_{T^* X}\times \pi_{\Spec R})_*\O_{T^*X\times \Spec R}$.

Since $\D_X^R$ is a sheaf of non-commutative rings, one should distinguish between left and right $\D_X^R$-modules.
We may also refer to a $\D_X^R$-module without specifying left or right if no confusion is possible. In these cases it is intended that the result holds in either case.

For any filtered $\D_X^R$-module $\M$ there is an associated sheaf of modules on $T^*X\times \Spec R$ given by
$(\pi_{T^*X}\times \pi_{\Spec R})^{-1}(\grrel\M)\allowbreak\otimes_{\pi^{-1}\grrel\D_X^R} \O_{T^*X\times \Spec R}$.
From now   we write $\grrel\D_X^R$ and $\grrel \M$ for the corresponding sheaves on $T^*X\times \Spec R$.

A filtration  compatible with $F_\bullet \sD_X^R$ on a $\D_X^R$-module $\M$ is said to be {\it good} if $\grrel\M$ is a coherent $\grrel\D_X^R$-module.
A quasi-coherent $\D_X^R$-module $\M$ locally admits a good filtration if and only if it is coherent \cite[Corollary D.1.2]{hotta2007d}, in fact one can take this filtration to be global \cite[Proof of Theorem 2.1.3]{hotta2007d}.
For a coherent $\D_X^R$-module $\M$ the support $\Chrel\M$ of $\grrel\M$ in $T^*X\times \Spec R$ is independent of the chosen filtration \cite[Lemma D.3.1]{hotta2007d} and is called the {\it relative characteristic variety}.
Equivalently, the relative characteristic variety is locally determined by the radical of the annihilator ideal of $\grrel \M$ in $\grrel\D_X^R$.

\begin{lemma}[{\cite[Lemma 3.2.2]{budur2020zeroI}}]\label{prop: SESBehaviourChrel}
  For any short exact sequence of coherent $\D_X^R$-modules
  $$0\to \M_1\to \M_2 \to \M_3 \to 0 $$
  it holds that $\Chrel\M_2 = \Chrel\M_1 \cup \Chrel\M_3.$
\end{lemma}

\begin{definition}\label{def: RelHolonomic}
  A coherent $\D_X^R$-module $\M$ is said to be {\it relative holonomic} if its relative characteristic variety is a finite union $\Chrel\M = \cup_w \Lambda_w \times S_w$ where $\Lambda_w\subseteq T^* X$ are irreducible conic Lagrangian subvarieties and $S_w\subseteq \Spec R$ are irreducible subvarieties.
\end{definition}

\begin{lemma}[{\cite[Lemma 3.2.4]{budur2020zeroI}}]\label{prop: SubquotientRelHol}
  Any subquotient of a relative holonomic module is relative holonomic.
\end{lemma}


The functor which associates to a left $\D_X^R$-module $\M$ the right $\D_X^R$-module $\M\otimes_{\O_X}\omega_X$ is an equivalence of categories, where $\omega_X$ is the canonical invertible sheaf.
The pseudoinverse associates $\Hom_{\O_X}(\omega_X,\allowbreak\M)$ to a given right-module $\M$.

Pick local coordinates $x_1,\ldots, x_n$ on $X$, that is, regular functions such that  $dx_1,\ldots,dx_n$ are a local basis for $\Omega_{X}^1$.
There is an induced local section $dx := dx_1\wedge \ldots \wedge dx_n$ for $\omega_X$. For any left $\D_X^R$-module $\M$ one  has a locally defined $\O_X\otimes R$-linear isomorphism $\M \to \M\otimes_{\O_X}\omega_X$  associating to any section $m$ the section $m^* = m dx$.
This can be made to commute with the $\D_X^R$-module structure.
That is, for any operator $P$ of $\D_X^R$ there is an {\it adjoint operator} $P^*$ such that
$$(P\cdot m)^* = m^* \cdot P^*.$$
Indeed, for a vector field $\xi = \sum \xi_i \partial_i$ this is satisfied by setting $\xi^* = -\sum \partial_i \xi_i$.
Iteration then extends to differential operators of arbitrary order, and $(PQ)^*=Q^*P^*$ for $P,Q\in \D_X^R$.

\subsection{Direct image}
Let $\mu:Y\to X$ be a morphism of varieties.
The {\it direct image functor} on right $\D_Y$-modules is defined by $$\mu_+ \M:= R\mu_*(\M\otimes_{\D_Y}^L \D_{Y\to X}) $$
where $\D_{Y\to X}:= \O_Y\otimes_{\mu^{-1}\O_X}\mu^{-1}\D_X$ is the transfer $(\D_{Y},\mu^{-1}\D_X)$-bimodule. 
There is an induced $\D_Y^R$-module direct image functor.
Indeed, consider a right $\D_Y^R$-module $\M$ and observe that multiplication by $r\in R$ is $\D_Y$-linear.
By the functoriality of the $\D_Y$-module direct image it follows that there is an associated endomorphism on $\mu_+ \M$.
This equips the direct image with a canonical structure of a complex of $\D_X^R$-modules.
For  $j\in \mathbb{Z}$, the cohomology sheaf $ H^j\mu_+\M$ is called the {\it $j$-th direct image}.

Whenever $\mu$ is proper and $\M$ is coherent as $\D_Y^R$-module it holds that $H^j\mu_+\M$ is coherent over $\D_X^R$ for any $j$.
The proof for this statement is identical to the absolute case \cite[Theorem 2.5.1]{hotta2007d}. The following proposition may be established identically to the absolute case \cite[Theorem 4.4.1]{sabbah2011introduction}.

\begin{proposition}\label{prop: DirectImageRelHol}
  Suppose that $\mu$ is proper and let $\M$ be a relative holonomic right $\D_Y^R$-module.
  Then $H^j\mu_+ \M$ is relative holonomic for any $j\in \mathbb{Z}$.
\end{proposition}

\subsection{Homological notions}\label{sec: HomNotion}
Let $n = \dim X$ and $r = \dim R$.
For some results in this section the distinction between left and right modules is relevant.
Such results have been stated in terms of right $\D_X^R$-modules, which is the case we will need. It should be clear that these results have obvious analogues for left $\D_X^R$-modules.

\begin{definition}
  Let $\M$ be a non-zero coherent $\D_X^R$-module. The smallest integer $j\geq 0$ such that $\Ex{\D_X^R}{j}(\M,\D_X^R) \neq 0$ is called the {\it grade} of $\M$ and is denoted $j(\M)$. If $\M = 0$ then $j(\M)$ is said to be infinite.
\end{definition}

\begin{definition}\label{def: BSIdeal}
  The {\it Bernstein-Sato-ideal} of a $\D_X^R$-module $\M$ is given by $B_\M:=\Ann_R \M$. We denote by $Z(B_\M)$ the zero locus of $B_\M$, that is, the reduced closed subscheme defined by the radical ideal of $B_\M$ in $\Spec R$.
\end{definition}

\begin{lemma}[{\cite[Lemma 3.4.1]{budur2020zeroI}}]\label{prop: ChrelGrades}
  Let $\M$ be a relative holonomic $\D_X^R$-module. Then
  $$\dim \Chrel \M + j(\M) = 2n + r.$$
\end{lemma}

\begin{lemma}[{\cite[Lemma 3.2.2]{budur2020zeroI}}]\label{rem: GradeIFFBernsteinIdeal} Let $\M$ be a relative holonomic $\D_X^R$-module. Then $Z(B_\M)$ is the projection of $\Chrel\M$ on $\Spec R$.
 Hence, $j(\M)=n+k$  if and only if $Z(B_\M)$  has codimension $k$ in $\Spec R$.
  \end{lemma}
  
\begin{definition}
  A non-zero coherent $\D_X^R$-module $\M$ is said to be {\it $j$-pure} if $j(\N) = j(\M) = j$ for every non-zero submodule $\N$.
\end{definition}

 \begin{lemma}[{\cite[Lemma 3.4.2]{budur2020zeroI}}]\label{prop: Injective3.4.2}
   Let $\M$ be a $j$-pure relative holonomic $\D_X^R$-module and suppose that $b\in R$ is not contained in any minimal prime ideal of $R$ containing $B_\M$. Then there exists a good filtration on $\M$ such that multiplication by $b$ induces injective endomorphisms on $\M$ and $\grrel\M$.
 \end{lemma}

\begin{corollary}\label{cor: injective}
Let $\M$ be a relative holonomic $\D_X^R$-module with $R=\bC[s_1,\ldots,s_r]$, $r>0$. 

(i) There exists a non-empty Zariski open subset $W(\M)$ of the space $R_1$ of polynomials in $R$ of degree one   
such that every $\ell \in W(\M)$ acts injectively on $\M$. 

(ii) One can assume, by shrinking $W(\M)$ if $W(\M)=R_1$, that there exists a  Zariski closed proper subset $V(\M)$ of $\bC^r$ such that  $$W(\M)=\{\ell\in R_1\mid \ell\text{ does not vanish on any irreducible component of }V(\M)\}.$$
\end{corollary}
\begin{proof}
Denote by $\M_i$ the largest submodule of $\M$ with $j(\M_i)\geq i\ge 0$. The modules $\M_i$ exist and
 form a decreasing sequence
$$\M=\M_0\supset \M_1\supset\dots$$
by \cite[IV.1.6.(i) and  IV.2.8]{B}.
By Lemma \ref{prop: SubquotientRelHol}, $\M_i$ are also relative holonomic. Thus by Lemma \ref{rem: GradeIFFBernsteinIdeal}, $j(\M_i)\ge n$ for all $i$ with $\M_i\neq 0$, and $\M_{i}=0$ if $i>n+r$.



The successive quotients $\M_i/\M_{i+1}$ are either $0$ or pure of grade $i$, by \cite[Proposition 4.11]{robin}. Let $I$ denote the set of indices $i$ such that $\M_i/\M_{i+1}\neq 0$. If $i\in I$, let $V_i\subset \bC^r$ be the zero locus of $B_{\M_i/\M_{i+1}}$, and let $W_i$ be the set of $\ell\in R_1$ that do not vanish on any irreducible component of $V_i$. Each $W_i$ is non-empty Zariski open in $R_1$ and every $\ell\in W_i$ acts injectively on $\M_i/\M_{i+1}$ by Lemma \ref{prop: Injective3.4.2}. If $i=n$, then $V_n=\bC^r$ and $W_n=R_1$. If $n< i\le n+r$, then $V_i\subsetneq \bC^r$ and $W_i$ might still be all of $R_1$. Set $W(\M):=\cap_{i\in I}W_i$. Then $W(\M)$ is non-empty Zariksi open in $R_1$ and  every $\ell\in W(\M)$ acts injectively on $\M$. This gives (i)

If $W(\M)\subsetneq R_1$,  define $V(\M):=\cup_{i}V_i$ where the union runs over $i\in I$ such that $W_i\neq R_1$. It is clear that this satisfies (ii). 
\end{proof}

\begin{corollary}\label{commutativity}
Let $\mu:Y\to X$ be a morphism of smooth varieties, and let $\M$ be a relative holonomic $\D_Y^{R}$-module, with $R=\C[s_1,\dots,s_r]$, $r>0$. There exists a finite set $J$ and a set of points $\{\beta_{i,j}\in\bC\mid 1\le i\le n, j\in J\}$ such that for every $\al$ in the non-empty Zariski open complement $\bC^r\setminus\cup_{i,j}\{z_i-\beta_{i,j}=0\}$,  the  natural  morphism  of $\sD_X$-modules 
\be\label{eqW} \left(H^0\mu_+ \M\right)\otimes _R R/\mathfrak{m}_\alpha \to H^0\mu_+(\M\otimes_R R/\mathfrak{m}_\alpha)\ee
is an isomorphism, where $\mathfrak{m}_\alpha=(s_1-\al_1,\ldots,s_r-\al_r)$ is the maximal ideal in $R$ of $\alpha$.
\end{corollary}

\begin{proof} Let $\gamma\in\bC^r$. First, let $\ell\in W(\M)\subset R_1$ be a polynomial of degree one, with $W(\M)$ as in  Corollary \ref{cor: injective}.  Then multiplication by $\ell$ on $\M$ followed by the direct image induces a long exact sequence of $\D_X^R$-modules
  $$0 \to H^0\mu_+\M \xrightarrow{\ell\cdot} H^0\mu_+ \M \to H^0\mu_+ \left(\M\otimes_{R} R/(\ell)\right) \to H^1\mu_+\M \xrightarrow{\ell\cdot} H^1\mu_+\M.$$
 Thus $(H^0\mu_+ \M)\otimes_{R} R/(\ell)$ is a $\sD_X$-submodule of $H^0\mu_+ (\M\otimes_{R} R/(\ell))$.
  Their quotient is isomorphic to the kernel of $\ell$ on $H^1\mu_+\M$. We can assume further that $\ell\in W(\M)\cap W(H^1\mu_+\M)$ since the intersection is Zariski open and dense. Then this kernel is zero, and hence
  $$\left(H^0\mu_+ \M\right)\otimes_{R} R/(\ell)\simeq H^0\mu_+(\M\otimes_{R} R/(\ell)).$$ 
By   Corollary \ref{cor: injective}, we can assume $W(\M)\cap W(H^1\mu_+\M)$ is the set of $\ell\in R_1$ that do not vanish on any irreducible component of a Zariski closed proper subset $V(\M)\cup V(H^1\mu_+\M)$ of $\bC^r$. Thus, there exists a finite set $J$ and a  set of points $\{\beta_{1,j}\in\bC\mid j\in J\}$ such that
 $\ell=s_1-\al_1\in W(\M)\cap W(H^1\mu_+\M)$ for $\al_1\in\bC\setminus\{\beta_{1,j}\mid j\in J\}$.

If $r=1$, the above argument gives the claim. If $r>1$, we proceed  by induction since $R/(s_1-\al_1) \simeq \bC[s_2,\ldots,s_r]$.
\end{proof}

 \section{Upper bounds}\label{secUB}
 
 We consider first the algebraic case of \Cref{thm: MainTheorem}. Since $B_F^a$ is the intersection of all local $B_{F,x}^a$, we may assume that $X$ is affine and admits local coordinates $x_1,\ldots, x_n$.
 Let $\mu$ be a strong log resolution of $f$ as in the introduction,  $G = F\circ\mu$, and let $g_j=f_j\circ\mu$.  As in the introduction, we use the notation $\sD_X[s]$ for $\sD_X^R$ if $R=\bC[s]$.

 \subsection{Translation to right modules}
 By the translation between left and right modules in \ref{sec: RelHol} the functional equation $P F^{s+a} = b(s) F^s$ may be restated as the equation
 $F^{s+a}dx \cdot P^* = b(s) F^s dx $
 in $$\N:=\D_X[s] F^s \otimes_{\O_X}\omega_X = F^sdx\cdot \D_X[s].$$
 Define $\M$ to be the submodule of $ \D_Y[s] G^s\otimes_{\O_Y}\omega_Y$ spanned by $G^s \mu^*(dx)$ over $\D_Y[s]$,
 $$
 \M:= G^s \mu^*(dx)\cdot \D_Y[s].
 $$
 \begin{lemma}
   The right $\D_Y[s]$-module $\M$ is relative holonomic.
 \end{lemma}
 \begin{proof}
   The left $\D_Y[s]$-module $\D_Y[s]G^s$ is relative holonomic by \cite[R\'esultat 1]{maisonobe2016filtration}.
   Then the associated right module $\D_Y[s]G^s\otimes_{\O_Y} \omega_Y$ is also relative holonomic.
   Hence, Lemma \ref{prop: SubquotientRelHol} implies that the submodule $\M$ is also relative holonomic.
 \end{proof}
 \subsection{$\D_X[s]\langle t\rangle$-modules}
 Let $\D_X[s]\langle t \rangle$ denote the sheaf of rings obtained from $\D_X[s]$ by adding a new variable $t$ which commutes with sections of $\D_X$ and is subject to $s_j t = t(s_j +a_j)$ for every $j=1,\ldots, r$ .
 The $\D_X[s]$-module $\N$ may be equipped with the structure of a right $\D_X[s]\langle t\rangle$-module by the action
 $$F^sdx \cdot P(x,\partial,s) \cdot t  =F^{s + a}dx\cdot P(x,\partial, s + a).$$
 In this formalism $B_F^a$ is the Bernstein-Sato ideal of $\N/ \N t$.
 An analogous $\D_Y[s] \langle t\rangle$-module structure can be given to $\M$.
 \begin{lemma}\label{lem: BernsteinSatoPolynomialUpstairs}
  The Bernstein-Sato ideal $B_{\M/\M t}$ contains a polynomial of the form
  $$b(s)=\prod_E \prod_{j=1}^N(\ord_{E}(g_1) s_1 + \cdots + \ord_{E}(g_r) s_r + k_E+ j)$$
  where $E$ ranges over the irreducible components of $\mu^*D$, for some $N\in \mathbb{Z}_{\geq 0}$.
 \end{lemma}
 \begin{proof}
  The proof is analogous to the one in   \cite[Section 4]{lichtin1989poles}. 
  We can reduce the claim to the local analytic Bernstein-Sato ideal at  a point $y\in Y$ lying in the support of $\mu^*D$,  since $B_{\M/\M t}$ is the intersection of the local analytic Bernstein-Sato ideals. Let $E_i$ with $i\in I$ be the local analytic irreducible components of $\mu^*D$ at $y$. We can assume that there are  local analytic coordinates $z_1,\ldots,z_n$ where every $E_i$ with $i\in I$ is determined by some $z_{j_i}$. After relabeling, we may be assumed that $j_i = i$. In these local coordinates
 $$G^s = \prod_{i\in I} z_i^{\sum_{j=1}^r   \ord_{E_i}(g_j)s_j}\qquad\text{ and }\qquad \mu^*(dx) = v \prod_{i\in I} z_i^{k_i} dz$$
 where $v$ is a local unit.
 Let $$P = v^{-1}\left(\prod_{i\in I}(-\partial_i)^{\sum_{j=1}^r a_j\ord_{E_i}(g_j)}\right) v.$$ Then
 $$G^{s+a}\mu^*(dx) \cdot P =  q(s) G^s \mu^*(dx) $$
 where
 $$q(s) = \prod_{i\in I}\left(\sum_{j=1}^r \ord_{E_i}(g_j)s_j + \sum_{j=1}^r a_j\ord_{E_i}(g_j) + k_i\right)\cdots\left(\sum_{j=1}^r \ord_{E_i}(g_j)s_j + 1 + k_i \right).$$
 \end{proof}

   The $\D_X$-linear endomorphism $t$ induces an endomorphism on $H^0\mu_+\M$.
  The relation $s_it = t(s_i +a_i)$ also holds on $H^0\mu_+\M$ due to the functoriality of the direct image.
  Hence $H^0\mu_+\M$ is equipped with the structure of a $\D_X[s]\langle t\rangle$-module.

  The surjection of right $\D_Y[s]$-modules $\D_Y[s]\to \M$ defined by $1\mapsto G^s \mu^*(dx)$ induces a morphism $H^0\mu_+ (\D_Y[s] )\to H^0\mu_+\M$.
  Observe that $H^0\mu_+ (\D_Y[s] ) = \mu_*(\D_{Y\to X}\otimes_\bC\bC[s])$ contains a global section corresponding to $1\otimes 1$.
  We write $u$ for the image of this section in $H^0\mu_+ \M$, and $\U$ for the right $\D_X[s]\langle t\rangle$-submodule generated by $u$.
  \begin{lemma}\label{lem: SurjectionUF}
  There is a surjective morphism of right $\D_X[s]\langle t \rangle$-modules $\U\to \N$ sending $u$ to $F^sdx$.
\end{lemma}
\begin{proof}
  This is analogous to the corresponding absolute result \cite[Chapter 5, p246]{bjork1979rings}. 
  One must  show that $(F^s dx)P = 0$ whenever $uP = 0$ for some differential operator $P$ over an open $V\subseteq X$.

  The resolution of singularities $Y\to X$ is an isomorphism over the complement of the divisor $D$ determined by $f$.
  This induces  isomorphisms $\U \simeq H^0\mu_+ \M \simeq  \N$  outside of $D$.
  It follows that the support of the coherent sheaf of $\O_V$-modules $\O_V ((F^s dx) P) $ lies in $D$.
  Thus $f^N ((F^s dx) P)  = 0$ for some sufficiently large $N\geq 0$.
  Note that $f$ is a non-zero divisor of $\N(V)$.
  Therefore, $(F^s dx) P= 0$ on $V$ as desired.
\end{proof}

\begin{lemma}\label{gradeBound}
The module $(H^0\mu_+\M)/\U$ is relative holonomic, and $j((H^0\mu_+\M)/\U)>n$.
\end{lemma}
\begin{proof} Let $\L=(H^0\mu_+\M)/\U$. By looking at the presentations for $\L$ and $\U$, since $\M$ is relative holonomic it follows that $\L$ is also relative holonomic. The fact that $j(\L)>n$ is equivalent to the fact that $\Ann_{\C[s]}\L\not=0$. By \cite[Theorem E]{robin}  for any $\alpha$ in the zero locus of $\Ann_{\C[s]}\L$,
\begin{equation}\label{tensor}
\L\otimes_{\C[s]}\frac{\C[s]}{\mathfrak{m}_\alpha}\not=0,
\end{equation}
where $\mathfrak{m}_\alpha$ is the maximal ideal corresponding to $\alpha$. Let $\C_\alpha=\C[s]/\mathfrak{m}_\alpha$. To prove that $\Ann_{\C[s]}\L\not=0$ it thus suffices that $\L\otimes_{\C[s]}\C_\alpha=0$ for some $\alpha\in\C^r$.

We consider the exact sequence of $\sD_X$-modules
\be\label{eqss}\U\otimes \C_\alpha\to \left(H^0\mu_+\M\right)\otimes\C_\alpha\to \L\otimes_{\C[s]}\C_\alpha\to 0.\ee
By Lemma \ref{commutativity}, $\left(H^0\mu_+\M\right)\otimes\C_\alpha=H^0\mu_+(\M\otimes\C_\alpha)$ for all $\al\in\bC^r$ outside a finite union of hyperplanes of type $\{z_i-\beta_{ij}=0\}$ with $\beta_{ij}\in\bC$. Among such $\al$, we pick now $\al\in\bZ^r$ satisfying that each $\al_i\ll 0$; for example, $\al=\al'-\bk$ with $\bk=(k,\ldots,k)\in\bZ^r$ for $k\in\bN$ arbitrarily large with respect to a fixed $\al'\in\al+\bZ^r$.

We consider the diagram
\[
\begin{tikzcd}
U=X\setminus D \arrow[r,"j'"]\arrow[dr,"j"]& Y\arrow[d,"\mu"]\\
& X
\end{tikzcd}
\]
where $j$ and $j'$ are the natural open embeddings. Since $\al_i\ll 0$ for each $i$, there is an  equality of regular holonomic right $\sD_Y$-modules
$$
\M\otimes_{\C[s]} \C_\al = (\sD_X[s]G^s\otimes_{\cO_Y}\omega_Y)\otimes_{\C[s]}\C_\al
$$
which can be checked locally. The last right $\sD_Y$-module corresponds to the regular holonomic left $\sD_Y$-module $\sD_Y[s]G^s\otimes_{\bC[s]}\bC_\al$. Moreover, there is an isomorphism of left $\sD_Y$-modules
 $$\sD_Y[s]G^s\otimes_{\bC[s]}\bC_\al\simeq
\sD_YG^\al  =\cO_Y[g^{-1}]= j'_+(\sD_Ug^{-1})$$
with $g=\prod_{j=1}^rg_j$, and their associated  de Rham complexes are isomorphic to the perverse sheaf $Rj'_*\bC_{U}[n]$,  see \cite[Theorem 2.5.1]{budur2020zeroI}. Since  $j$ is an affine morphism, the derived direct image $R\mu_*(Rj'_*\bC_{U}[n]) = Rj_*\bC_U[n]$ is also perverse and hence equal to the perverse $0$-direct image ${}^pR^0\mu_*(Rj'_*\bC_{U}[n])$. Equivalently, using the Riemann-Hilbert correspondence between regular holonomic $\sD$-modules and perverse sheaves, there is an isomorphism of left $\sD_X$-modules,
$$ 
(H^0\mu_+)\cO_Y[g^{-1}] \simeq \cO_X[f^{-1}] \simeq \D_XF^s \otimes_{\bC[s]} \bC_\al.
$$
In terms of right $\sD_X$-modules this gives, 
$$H^0\mu_+(\M\otimes_{\bC[s]}\bC_\al) \simeq \N\otimes_{\bC[s]}\bC_\al \simeq F^\al dx\cdot\sD_X.$$
Thus the first map in (\ref{eqss})  is the map $\U\otimes \C_\alpha=\D_{X}[s]u\otimes \C_\alpha\to F^\al dx\cdot\sD_X$  that sends $u$ to $F^\alpha dx$. Hence this map is surjective. This shows that $\L\otimes_{\C[s]}\C_\alpha=0$ as required.
\end{proof}

\subsection{Proof of Theorem \ref{thm: MainTheorem} - algebraic case.}  Let $\L=(H^0\mu_+\M)/\U$.
  The Bernstein-Sato ideals $B_{\L t^n}$ form an increasing sequence of ideals in the Noetherian ring $\C[s]$.
  Hence there must exist some $N\geq 1$ such that $B_{\L t^n} = B_{\L t^{n+1}}$ for all $n\geq N$. 

  By \Cref{gradeBound} the $\D_X[s]$-module $\L$ has grade   $\ge n + 1$, so Lemma \ref{rem: GradeIFFBernsteinIdeal} provides some non-zero $q(s_1,\ldots,s_r) \in B_{\L}$.
	Then also $q\in B_{\L t^N}$.
  Observe that one has the  relation
  $$q(s_1,\ldots,s_r)t = tq(s_1+a_1,\ldots, s_r + a_r).$$
  In particular it follows that $q(s + a)\in B_{\L t^{N+1}}$.
  Due to the stabilisation $B_{\L t^N} = B_{\L t^{N+1}}$ it follows by iteration that $q(s+ja)\in B_{\L t^N}$ for any integer $j\geq 0$.
  Due to the estimate for the slopes in \Cref{thrmMoreA} it follows that we can pick some polynomial $r(s)$ which annihilates $\L t^N$ and such that $r(s+a)$ does not vanish on any codimension one irreducible component of $Z(B_{F,x}^a)$.

We now follow closely \cite{kashiwara1976b} and \cite{lichtin1989poles}.
  Let $b(s)$ be the Bernstein-Sato polynomial for $\M/\M t$ provided by \Cref{lem: BernsteinSatoPolynomialUpstairs}. Notice that the action of $t$ is injective on $\M$. This means that the morphism
  $$\phi:\M\to \M: m_1\mapsto \text{the unique }m_2 \text{ such that }m_1b(s)=m_2t,$$
  is well-defined and $\D_Y$-linear, and that $b(s)=t\circ \phi:\M\to\M$ as a morphism of $\D_Y$-modules.  By functoriality we thus conclude that $b(s)=t\circ H^0\mu_+\phi$ as a morphism on $H^0\mu_+\M$. This implies that
  $$(H^0\mu_+\M) b(s) \subset H^0(\mu_+\M) t.$$
  Set $B := \prod_{j=0}^{N} b(s+ ja)$. Then with a similar argument applied  inductively we have that $(H^0\mu_+\M) B(s)\subset (H^0\mu_+\M) t^{N+1}$. Thus have
  $$(H^0\mu_+\M) B(s)r(s+a)\subset(H^0\mu_+\M) t^{N+1}r(s+a)=(H^0\mu_+\M) t^{N}r(s)t.$$
 Since $\L t^N=(H^0(\mu_+\M)t^N+\U)/\U$ and $r$ annihilates $\L t^N$, we have
 $$((H^0\mu_+\M)t^N+\U)r\subset \U,$$ 
 and hence $$(H^0\mu_+\M)t^Nr\subset ((H^0\mu_+\M)t^N+\U)r\subset \U.$$
In particular, since $\U\subset H^0\mu_+\M$ we have that $\U B(s)r(s+a)\subset \U t$, that is, $B(s)r(s+a)$ lies in the Bernstein-Sato ideal $B_{\U/\U t}$.

	By \Cref{lem: SurjectionUF}  we have a $\D_X[s]\langle t\rangle$-linear surjection $\U \to \N$. Thus $\U/\U t$ surjects onto $\N/\N t$, and so $B(s)r(s+a)$ also annihilates $\N/\N t$.
	This implies that $Z(B_F^a)\subseteq Z(B(s)r(s+a))$. Since we know that none of the irreducible components of $Z(r(s+a))$ are irreducible components of codimension one of $Z(B_F^a)$, this gives the desired result. \hfill $\Box$

\subsection{The analytic case}
The proof of \Cref{thm: MainTheorem} proceeds similarly in the local analytic case, that is, when the smooth affine variety $X$ is replaced with the germ of a complex manifold $(X,x)$, or equivalently, with a very small open ball $\Omega_x$ centered at $x$ in $X$. 
By \cite[3.6]{budur2020zeroI}, all the results we have used for relative holonomic $\sD_X$-modules hold in the local analytic case. 
The log resolution $\mu:Y\to X=\Omega_x$ has the property that $Y$ admits a finite cover $\{Y_k\}$ of open subsets such that $g=f\circ\mu$ is a locally a monomial. Relative holonomicity can be defined for any analytic $\sD_Y[s]$-module admitting a good filtration on each $Y_i$, and by \cite[Theorem 1.17]{AnalyticDirectIm} the analytic direct image functor $\mu_+$ for such modules preserves relative holonomicity. Thus all the results from this section extend to the analytic version.

\begin{remark}
One lacks a bound on $c$ in \Cref{thm: MainTheorem} since $\N$ in the above proof is difficult to control.
\end{remark}

\section{Lower bounds}\label{secLB}

\subsection{Proof of Proposition \ref{prop1.3}.} Let $x$ be a smooth point of $C$. We can assume that $x_1$ is a local equation for $C$ at $x$. Then locally at $x$, $f_j=x_1^{N_j}u_j$ with $N_j=\ord_C(f_j)$ and $u_j$ a locally invertible function. We assume $m=\sum_{j=1}^rN_ja_j$ is non-zero. One easily computes now that $B_{F,x}^a$ is the principal ideal generated by $$b(s)=\prod_{c=1}^m \left(\left(\sum_{j=1}^r N_js_j\right)+ c \right)$$
corresponding to the relation $$b(s)\prod_{j=1}^rf_j^{s_j}= \pa_1^m(\prod_{j=1}^r u_j^{a_j})^{-1}\prod_{j=1}f_j^{s_j}.$$  \hfill $\Box$

\subsection{Jumping walls}\label{sec: JumpingWall}
In this subsection we establish \Cref{thm: JumpingWall} on the relation between the jumping walls and $Z(B_F^a).$
By \cite[Corollary 3.12]{kollar1997singularities} one can rephrase the $\LCT$-polytope and $\KLT_a$-region  as
\begin{align*}
    \LCT(F) &=  \{\lambda\in \mathbb{R}_{\geq 0}^r : (X,F^\lam) \text{ is log-canonical}\}\\
    \KLT_{a}(F)&= \{\lam\in \mathbb{R}_{\geq 0}^r: (X,F^{\lam - a}) \text{ is Kawamata log-terminal}\}.
\end{align*}
For our purposes the analytical reformulation of Kawatama log-terminality from  \cite[Proposition 3.20]{kollar1997singularities} is the most convenient,
$$\KLT_a(F) = \{\lambda\in \mathbb{R}_{\geq 0}^r : \prod_{j=1}^r \abs{f_j}^{-2(\lambda_j-a_j)}\text{ is integrable near any }x\in X\}.$$
Similarly, the stalk of the mixed multiplier ideal sheaf $\mathcal{J}(F^\lambda)$ for $\lambda\in \mathbb{R}^r_{\geq 0}$ at  any $x\in X$ is
$$\mathcal{J}(F^\lambda)_x = \{\phi \in \O_{X,x}: \abs{\phi}^2 \prod_{j=1}^r \abs{f_j}^{-2\lambda_j} \text{ is integrable near }x\}.$$
Let $\mu$ be a strong log resolution of $f$ as in the introduction,  $G = F\circ\mu$, and let $g_j=f_j\circ\mu$.

\begin{proof}[Proof of Theorem \ref{thm: JumpingWall}.] Let $E$ be an irreducible component of $\mu^*D$. Suppose that the hyperplane $\{\sum_{j=1}^r \ord_E(g_j)s_j = k_E + c\}$ for some $c\in\bZ_{>0}$ is the affine span of a facet $\sigma$ of a jumping wall of $F$ which intersects $\KLT_{a}(F)$. 
  We show that $\sum_{j=1}^r \ord_E(g_j)s_j + k_E + c = 0$ determines an irreducible component of $Z(B_F^a)$.

Note that the facet $\sigma$ must be included in $\KLT_{a}(F)$. Let $\lambda$ be a point of  $\sigma$. Then there must exist some $x\in D$ and $\phi\in \O_{X,x}\setminus \mathcal{J}(F^\lam)_x$ such that
  $$\int \abs{\phi}^2\prod_{j=1}^r\abs{f_j}^{-2(\lambda_j -\varepsilon_j)}\psi dxd\bar x <  \infty, \qquad \int \prod_{j=1}^r \abs{f_j}^{-2(\lambda_j - a_j)}\psi dxd\bar x< \infty.$$
  for any $\varepsilon \in \mathbb{R}_{>0}^r$ and positive bump function $\psi$ supported on a sufficiently small neighbourhood of $x$, where $dx=dx_1\ldots dx_n$ for local coordinates $x_1,\ldots, x_n$ on $X$.

  Pick some $b(s) \in B_F^a$ and take the support of $\psi$ to be sufficiently small such that there exists some local differential operator $P$ with $b(s)F^s = PF^{s+a}$.
  By conjugation it follows that $\overline{b}(s) \overline{F}^s = \overline{P}\overline{F}^{s+a}$.
  Holomorphic and antiholomorphic differential operators commute so
  $$\abs{b(s)}^2 \prod_{j=1}^r \abs{f_j}^{2s_j} = P\overline{P} \abs{f_j}^{2(s_j + a_j)}.$$
  Now assume that the real part of all $2(s_j+a_j)$ is strictly greater than the order of $P$. Then $\abs{f_j}^{2(s_j+a_j)}$ has enough continuous derivatives to apply integration by parts.
  This yields that
  $$\abs{b(s)}^2 \int \prod_{j=1}^r\abs{f_j}^{2s_j}|\phi|^2 \psi dxd\bar x= \int \prod_{j=1}^r  \abs{f_j}^{2(s_j+a_j)} P^*\overline{P^*}\abs{\phi}^2\psi dxd\bar x$$
  View this as an equality of meromorphic functions of $s$ to conclude that the equality holds for arbitrary $s\in \mathbb{R}^r$ provided both integrals are finite.

  Now take $s = -\lambda + \varepsilon$ and let $\varepsilon$ tend to zero from above.
  Then, by dominated convergence, the integral on the right hand side converges to a finite number.
  On the other hand, since $\phi$ is not in $\mathcal{J}(F^s)_x$, the integral on the left hand side tends to infinity by the monotone convergence theorem.
  This means that the equality is only possible if $b(s)$ vanishes on $(-\lambda_1,\ldots, -\lambda_r)$.
  Since the point $\lambda$ is arbitrary on $\sigma$,  and $b(s)\in B_F^a$ is also arbitrary, we conclude that $\sum_j \ord_{E}(g_j)s_j +  k_E + c=0$ determines an irreducible component of $Z(B_F^a)$.
\end{proof}

\begin{proof}[Proof of Corollary \ref{thm: LCT}.]   A facet of $\LCT(F)$ is by definition a facet of a jumping wall of $F$. By \Cref{thm: JumpingWall} it is enough to show that  $\sum_{j=1}^r \ord_E(g_j)s_j = k_E + 1$ intersects $\KLT_a(F)$.  Let $\lam$ be an interior point of this facet of $\LCT(F)$. It is enough to show $F^{\lam-a}$ is Kawamata log-terminal. Let $E'$ be an irreducible component of $\mu^*D$. Then $\sum_j\ord_{E'}(f_j)\lam_j\le k_{E'}+1$. Equality holds if and only if $E'$ determines the same facet of $\LCT(F)$ as $E$, that is, 
$$
\left\{\sum_{j=1}^r \ord_{E'}(g_j)s_j = k_{E'} + 1\right\} =\left\{\sum_{j=1}^r \ord_E(g_j)s_j = k_E + 1\right\}.
$$
Let $I_E$ be the set of such $E'$. By  assumption,  there exists at least one $j$ with $ \ord_{E}(f_j)\cdot a_j\neq 0$. This implies that for the same $j$, the same holds for $E'\in I_E$. Thus for $E'\in I_E$ we have  $\sum_j \ord_{E'}(f_j) a_j> 0$ since $a\in\bN^r$. Hence for all irreducible components $E'$ of $\mu^*D$ one has
$$
\sum_j\ord_{E'}(f_j)(\lam_j-a_j)<k_{E'}+1
$$
as claimed.
\end{proof}

\subsection{Real jumping walls}\label{sec: RealJump}
Finally, we establish the real analogues for the results in \ref{sec: JumpingWall}. As mentioned in the introduction, one of the motivations is that this gives sometimes a different way of producing irreducible components of the zero loci of Bernstein-Sato ideals, another motivation being the potential applications to statistics.

Let $X_\mathbb{R}$ be a real affine algebraic manifold. Let $F = (f_1,\ldots,f_r)$ be a tuple of real algebraic functions on $X_{\mathbb{R}}$.
Fix  $a\in \mathbb{Z}_{\geq 0}^r$ and assume that $\prod_{j=1}^r f_j^{a_j}$ is not invertible.

The  Bernstein-Sato ideal $B_{F}^a\subset \bR[s]$, with $s=s_1,\ldots, s_r$,
consists by definition of all polynomials $b(s)\in \mathbb{R}[s]$ such that
$$b(s)F^s \in \D_{X_\mathbb{R}}[s]F^{s+a}$$
where $\D_{X_\mathbb{R}}$ denotes the ring of real algebraic differential operators on $X_\mathbb{R}$. 
If $F_\bC$ is the complexification on $X_\bC=X_\bR\otimes_\bR \bC$ of $F$, it is easy to see that $B_{F}^a$ consists of all polynomials obtained by replacing the coefficients of $q(s)$ with their real parts for all $q(s)\in B_{F_\bC}^a$. It is conjectured in \cite{B-ls} that $B_{F_\bC}^a$ is generated by polynomials with coefficients in $\bQ$, in which case the same polynomials would generate $B_F^a$. Since this conjecture is open, for now we can only conclude from Theorem \ref{thrmMoreA} the following:

\begin{lemma}\label{lemXr} Let $X_\mathbb{R}$ be a real affine algebraic manifold. Let $F = (f_1,\ldots,f_r)$ be a tuple of real algebraic functions on $X_{\mathbb{R}}$. Let $F_\bC$ be the $F$ considered as having complex coefficients. Fix  $a\in \mathbb{Z}_{\geq 0}^r$ and assume that $\prod_{j=1}^r f_j^{a_j}$ is not invertible. Then the codimension-one part of $Z(B_{F_\bC}^a)$ in $\bC^r$ consists of the complexification of the real codimension-one part of the  zero locus  $Z(B_F^a)$ in $\bR^r$.
\end{lemma}

A similar comparison holds between the local Berstein-Sato ideals $B_{F,x}^a$ and $B_{F_\bC,x}^a$
for $x\in X_\bR$,  where $B_{F,x}^a$ consists of all polynomials $b(s)\in \mathbb{R}[s]$ such that
$$b(s)F^s \in \D_{X_{\mathbb{R}},x}[s]F^{s+a}$$
where $\D_{X_\mathbb{R},x}[s]$ denotes the ring of germs at $x$ of real analytic differential operators on $X_\mathbb{R}$. Moreover, as in the complex affine case,  $B_F^a$ is the intersection of $B_{F,x}^a$ for $x\in X_\bR$.

Denote by $\O_{X_\mathbb{R}}$  the sheaf of real analytic functions on $X_\mathbb{R}$. After Saito \cite{RealLogCan}, we define the {\it real mixed multiplier ideals sheaves}  $\cJ_\bR(F^\lam)\subset \cO_{X_\bR}$ for $\lam\in\bR^r_{\ge 0}$ by setting 
$$
\mathcal{J}_{\mathbb{R}}(F^\lambda)(U): =\left\{\phi\in \cO_{X_\bR}(U) :  \abs{\phi}\prod_{j=1}^r \abs{f_j}^{-\lambda_j} \text{ is locally integrable on }U\right\}.
$$

Let $\mu:Y_\mathbb{R}\to X_\mathbb{R}$ be a real log resolution of singularities for $f:=\prod_{j=1}^rf_j$, that is, $\mu^* f$ and $\mu^*dx_1\ldots dx_n$ are locally monomial up to multiplication by an invertible function, where $x_1,\ldots, x_n$ are local algebraic coordinates on $X_\bR$. Since $X_\bR$ is assumed to be the underlying real analytic manifold of a smooth scheme $X$ defined over $\bR$, $Y_\bR$ is the underlying real analytic manifold of a smooth scheme $Y$ obtained by blowing up $X$ successively along smooth centers defined over $\bR$. Then the components of the divisor determined by $\mu^*f$ in $Y_\bR$ are the non-empty real loci of the components of the divisor defined by $f$ in $Y$, see \cite[1.2]{RealLogCan}. As before, we denote $k_{E}:=\ord_{E}(\det \operatorname{Jac}(\mu))\in\bN$ for the order of vanishing of the determinant of the Jacobian of $\mu$ along an irreducible component $E$ of the simple normal crossings divisor determined by $\mu^*f$ in $Y_\bR$.

Fix some $x\in X_\bR$ with $f(x)=0$. Associated to $\lambda\in\bR^r_{\ge 0}$ is the region $$\mathcal{R}_{\mathbb{R},F,x}(\lambda) := \{\lambda' \in \mathbb{R}_{\geq 0}^r : \mathcal{J}_\mathbb{R}(F^\lambda)_x \subseteq \mathcal{J}_\mathbb{R}(F^{\lambda'})_x\}.$$
The {\it real jumping wall at $x$} associated to $\lambda$ is the intersection of the boundary of $\mathcal{R}_{\mathbb{R},F,x}(\lambda)$ with $\mathbb{R}_{>0}^r$. 
The {\it $\RLCT$-polytope at $x$}  is the closure $\RLCT_x(F)$ of $\mathcal{R}_{\mathbb{R},F,x}(0)$.
The stalk $\mathcal{J}_{\mathbb{R}}(F^\lam)_x$ admits a characterization similar to the complex case, see \cite[Proposition 1]{RealLogCan}.
It follows that the facets of the jumping wall are cut out by hyperplanes of the form $\sum_{j=1}^r\ord_E(g_j)s_j = k_E + c$ with $c\in \mathbb{Z}_{>0}$ and the $\RLCT$-polytope is cut out by hyperplanes of the form $\sum_{j=1}^r\ord_E(g_j)s_j = k_E + 1$. 
Here, $E$ runs over all irreducible components of the divisor determined by $\mu^* f$ with $x\in\mu(E)$.
The {\it $\RKLT_{a}$-region} is defined by
$$\RKLT_{a,x}(F) := \{\lambda\in \mathbb{R}_{\geq 0}^r : \prod_{j=1}^r \abs{f_j}^{-(\lambda_j-a_j)}\text{ is integrable near }x\}.$$
The following theorem now follows similarly to \Cref{thm: JumpingWall}.

\begin{theorem}\label{thmRs} With the assumptions as in Lemma \ref{lemXr}, and with $x\in f^{-1}(0)\subset X_\bR$, if a facet of a real jumping wall of $F$ at $x$ intersects $\RKLT_{a,x}(F)$, then it determines an irreducible component of $Z(B_{F,x}^a)$.
\end{theorem}
\begin{proof}
Let  $\sigma$ be a facet of a real jumping wall of $F$ at $x$ which intersects $\RKLT_{a,x}(F)$. The affine span of $\sigma$ must be a hyperplane of the form $\sum_{j=1}^r \ord_E(g_j)s_j = k_E + c$ with $c\in\bZ_{>0}$, where $E$ is an irreducible component of the divisor determined by $\mu^* f$ with $x\in\mu(E)$. We show that $\sum_{j=1}^r \ord_E(g_j)s_j + k_E + c = 0$ determines an irreducible component of $Z(B_{F,x}^a)$.

  Let $\lambda$ be a point on $\sigma$. Then there must exist  $\phi\in \O_{X_\bR,x}\setminus \mathcal{J}_\mathbb{R}(F^s)_x$ such that
  $$\int \abs{\phi}\prod_{j=1}^r\abs{f_j}^{-(\lambda_j -\varepsilon_j)}\psi dx <  \infty, \qquad \int \prod_{j=1}^r \abs{f_j}^{-(\lambda_j - a_j)}\psi dx< \infty.$$
  for any $\varepsilon \in \mathbb{R}_{>0}^r$ and positive bump function $\psi$ supported on a sufficiently small neighbourhood of $x$, where $dx=dx_1\ldots dx_n$ and $x_1,\ldots, x_n$ are local coordinates on $X_\bR$ at $x$.

  Pick some $b(s) \in B_{F,x}^a$ and take the support of $\psi$ to be sufficiently small such that there exists some local differential operator $P\in\sD_{X_\bR,x}[s]$ with $b(s)F^s = PF^{s+a}$.
  Assume that the specialization of $s_j+a_j$ to a complex number has real part strictly greater than the order of $P$ for all $j$.
  Then $b(s) \prod_{j=1}^r \abs{f_{j}}^{s_j} = P \prod_{j=1}^r \abs{f_{j}}^{s_j +a_j}$ and $\abs{f_{j}}^{s_j +a_j}$ has enough continuous partial derivatives to apply integration by parts.
  This yields that
  $$b(s)\int  \prod_{j=1}^r \abs{f_{j}}^{s_j} \abs{\phi}\psi dx= \int \prod_{j=1}^r \abs{f_{j}}^{s_j +a_j} P^* \abs{\phi}\psi dx.$$
  View this as an equality of meromorphic functions in $s$ to deduce that the equality holds for arbitrary $s\in \mathbb{R}^r$ provided both integrals are finite.

  Now take $s = -\lambda + \varepsilon$ and let $\varepsilon$ tend to zero from above.
  Then, by dominated convergence, the integral on the right hand side stays finite as $\varepsilon$ tends to zero.
  On the other hand, by monotone convergence, the integral on the left hand side tends to infinity since $\phi$ is not in $\mathcal{J}_\mathbb{R}(F^s)_x$.
  This means that $b(s)$ vanishes on $(-\lambda_1,\ldots, -\lambda_r)$.
  Since the point $\lambda$ on the facet $\sigma$ and $b(s)\in B_{F,x}^a$ were arbitrary we conclude that $\sum_j \ord_{E}(g_j)s_j +  k_E + c=0$ determines an irreducible component of $Z(B_{F,x}^a)$.
\end{proof}

Precisely as with \Cref{thm: LCT} one obtains:

\begin{corollary}\label{corRs} With the same assumptions as in Theorem \ref{thmRs},
  suppose that the equation $\sum_{j=1}^r \ord_E(g_j)s_j = k_E + 1$ defines the affine span of a facet of $\RLCT_x(F)$. If $a_j\neq 0$ and $\ord_{E}(g_j)\neq 0$ for some $j$, then $\sum_j \ord_{E}(g_j)s_j +k_E + 1=0$ defines an irreducible component of $Z(B_{F,x}^a)$.
\end{corollary}

\section{Example}\label{secEx}

Let $f_1 =  y^2 -x^2 + x^3$ and $f_2 = y$ define the coordinate functions of the morphism $F:\C^2 \to \C^2$.
We compare the Bernstein-Sato zero locus $Z(B_F^a)$ for $a = (1,2)$ with the estimates we obtained in this article. Using  the library \texttt{dmodideal.lib} \cite{LLM}  in \texttt{SINGULAR} \cite{DGPS} yields the principal ideal $$B_F^a = \left((s_1+1)(s_2+1)(s_2+2) \prod_{l=2}^5(2s_1 +  s_2 + l)\right).$$
A strong log resolution $\mu:Y\to X$ may be found by use of one blowup. Let $E_j$ be the strict transform of $f_j=0$ for $j=1,2$, and let $E_0$ be the exceptional divisor.
Then \Cref{thm: MainTheorem} yields that
$$Z(B_F^a) \subseteq \bigcup_{l=1}^\infty Z(s_1 + l) \cup Z(s_2 + l) \cup Z(2s_1 + s_2 + l+ 1).$$
The trivial estimate \Cref{prop1.3} yields that
$$Z(s_1 + 1)\cup Z(s_2 + 1) \cup Z(s_2 + 2) \subseteq Z(B_F^a).$$

\noindent
We have
$$
  \KLT_a(F)= \{\lambda\in \mathbb{R}_{\geq 0}^r : \lambda_1 <2, \ \lambda_2 <3, \text{ and }\ 2\lambda_1 + \lambda_2 <6\},
$$
$$
\LCT(F) = \{\lambda\in \mathbb{R}_{\geq 0}^r :\lam_1\le 1, \lam_2\le 1, \text{ and }2\lam_1+\lam_2\le 2 \},
$$
see Figure \ref{fig1}.
Further, a polynomial $h\in\bC[x,y]$ belongs to the ideal $\mathcal{J}(F^\lambda)$ if and only if
$$\ord_{E_{1}}(h)\geq \lambda_1,\quad  \ord_{E_{2}}(h)\geq \lambda_2, \quad\text{ and }  \ord_{E_{0}}(h)\geq 2\lambda_1 + \lambda_2 -1.$$ Then 
\begin{align*}
\cJ(F^\lam) &=\bC[x,y]\quad\text{ for } \lam\in \LCT^o(F):=\LCT(F)\setminus (\{\lam_2=1\}\cup\{2\lam_1+\lam_2=2\}),\\\cJ(F^\lam) &=(x,y)\quad\text{ for }\lam\in [0,1)^2\setminus \LCT^o.
\end{align*}
By translating these two regions by integral vectors $(m_1,m_2)\in\bN^2$ one obtains the other regions of constancy of mixed multiplier ideals, the latter equal to  $(f_1^{m_1}f_2^{m_2})$ and, respectively, $(x,y)f_1^{m_1}f_2^{m_2}$. The jumping walls are depicted in Figure \ref{fig1}.  

All  irreducible components of $Z(B_F^a)$ arise from the facets of the jumping walls in this example. Hence the lower bound for $Z(B_F^a)$ following from \Cref{thm: JumpingWall} is tight. In this case the estimates coming from the real jumping walls at the origin are identical to the foregoing estimates.

\begin{figure}[H]
  \centering
  \begin{minipage}{.4\textwidth}
      \centering
      \includegraphics[width =0.5 \textwidth]{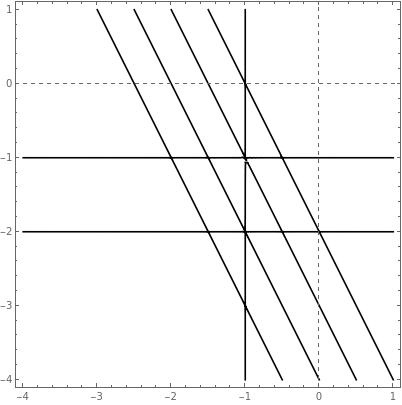}
  \end{minipage}
  \begin{minipage}{.4\textwidth}
      \centering
      \includegraphics[width = 0.5\textwidth]{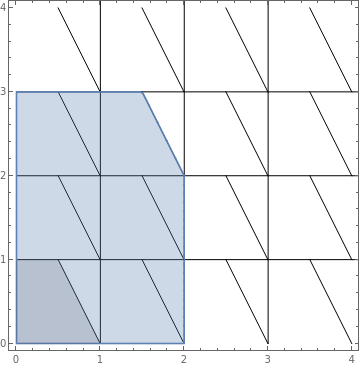}
  \end{minipage}
  \captionsetup{justification=centering}
  \caption{Left: $Z(B_F^a)$. Right: The jumping walls of $F$ with $\KLT_a(F)$ lightly shaded, and $\LCT(F)$ in darker shade.}
  \label{fig1}
\end{figure}

\end{document}